\documentclass[a4paper,12pt]{amsart}

\usepackage[english]{babel}
\usepackage{amsmath,amsthm,amsfonts,amssymb,units}
\usepackage[sort]{cite}
\usepackage{xcolor}

\theoremstyle{plain}
\newtheorem{theo}{Theorem}[section]
\newtheorem{lem}[theo]{Lemma}

\newtheorem{rem}[theo]{Remark}

\theoremstyle{definition}
\newtheorem{ex}{Example}
\numberwithin{equation}{section}

\newcommand{\brac}[1]{\left(#1\right)}

\def\om{\Omega}

\def\p{\partial}
\def\reals{\mathbb{R}}
\def\id{{\rm id}}
\def\eps{\varepsilon}
\def\epst{\tilde{\eps}}
\def\A{\alpha}
\def\At{\tilde{\A}}
\def\cf{c_{\rm f}}
\def\cfA{c_{{\rm f},\A}}
\def\cfe{c_{{\rm f},\eps}}
\def\cfAt{\tilde{c}_{{\rm f},\A}}
\def\cp{c_{\rm p}}

\def\cme{c_{{\rm m},\eps}}
\def\dx{{\rm d}x}
\def\dy{{\rm d}y}
\def\ut{\tilde{u}}
\def\xt{\tilde{x}}
\def\opluse{\oplus_{\eps}}

\def\ul{\underline}
\def\ol{\overline}

\newcommand{\inv}[1]{{#1}^{-1}}

\newcommand{\set}[2]{\{#1\,\mid\,#2\}}
\newcommand{\bset}[2]{\bigg\{#1\,\bigg|\,#2\bigg\}}

\def\na{\nabla}
\DeclareMathOperator{\opdiv}{div}
\def\div{\opdiv}
\DeclareMathOperator{\rot}{rot}
\DeclareMathOperator{\curl}{curl}
\DeclareMathOperator{\diam}{diam}

\DeclareMathOperator{\cont}{\sf C}
\DeclareMathOperator{\lebesgue}{\sf L}
\DeclareMathOperator{\hilbert}{\sf H}
\DeclareMathOperator{\divergence}{\sf D}
\DeclareMathOperator{\rotation}{\sf R}

\def\cic{\mathring{\cont}{}^\infty}

\def\lt{\lebesgue^2}					
\def\li{\lebesgue^\infty}			
\def\ltz{\lt_0}						

\def\ho{{\hilbert^1}}				
\def\hoc{\mathring{\hilbert}{}^1}

\def\d{{\divergence}}				
\def\dz{\d_0}
\def\dc{\mathring{\d}}
\def\dcz{\dc_0}

\def\r{{\rotation}}					
\def\rz{\r_0}
\def\rc{\mathring{\r}}
\def\rcz{\rc_0}

\newcommand{\abs}[1]{\left|#1\right|}	

\newcommand{\norm}[1]{|#1|}
\newcommand{\normltom}[1]{\norm{#1}_{\lt}}
\newcommand{\normltomw}[2]{\norm{#1}_{\lt,#2}}
\newcommand{\normltomwi}[2]{\norm{#1}_{\lt,\inv{#2}}}

\newcommand{\scp}[2]{\langle#1,#2\rangle}
\newcommand{\scpltom}[2]{\scp{#1}{#2}_{\lt}}
\newcommand{\scpltomw}[3]{\scp{#1}{#2}_{\lt,#3}}

\title[\sc A Short Note on a Weighted Friedrichs Inequality]
{\Large\sf A Short Note on a Weighted Friedrichs Inequality}

\author{Immanuel Anjam}
\address{}
\email[Immanuel Anjam]{immanuel.anjam@gmail.com}

\author{Dirk Pauly}
\address{Faculty of Mathematics,
University of Duisburg-Essen, Campus Essen, Germany}
\email[Dirk Pauly]{dirk.pauly@uni-due.de}

\date{4 March 2019}

\begin{document}

\begin{abstract}
In this note we derive an upper bound for the constant $\cfA>0$ in the weighted Friedrichs type inequality
\begin{equation*}
	\forall \varphi\in\hoc(\om) \qquad
	\normltom{\varphi} \leq \cfA \sqrt{\scpltom{\A\na\varphi}{\na\varphi}} ,
\end{equation*}
where $\om\subset\reals^d,d\ge1$ is a bounded domain, and $\A$ a bounded, self-adjoint, and uniformly positive definite matrix valued function. The contents of this note follow in a straightforward manner from well known results. In particular, for a constant diagonal matrix $\A$ we obtain the bound
\begin{equation*}
	\cfA \leq
	\brac{\pi\sqrt{\frac{\A_1}{l_1^2}+\cdots+\frac{\A_d}{l_d^2}}}^{-1} ,
\end{equation*}
where $l_i$ are the side lengths of a $d$-interval encompassing $\om$, and $\A_i$ are the diagonal entries of $\A$. Extensions to cases of unbounded domains and partial homogeneous boundary conditions are remarked upon. We also apply the main result in a posteriori error estimation for an elliptic problem and present some numerical results. Lastly, we use the main result to derive an improved upper bound of the tangential Maxwell constant for convex domains.
\end{abstract}

\maketitle
\tableofcontents
\newpage


\section{Introduction}
\label{sec:INTRO}

We denote by $x:=(x_1,\ldots,x_d)$ the Euclidean coordinates in $\reals^d,d\ge1$, and by $\om\subset\reals^d$ a bounded domain. The calculations performed in this note are invariant with respect to translations of the domain, so without loss of generality we assume $\om$ to be contained in the open $d$-interval
\begin{equation*}
	I:=\prod_{i=1}^{d} (0,l_i), \qquad 0<l_i<\infty .
\end{equation*}

The space of smooth scalar- or vector-valued functions vanishing on the boundary of the domain is denoted by $\cic(\om)$. We denote by $\scpltom{\,\cdot\,}{\,\cdot\,}$ and $\normltom{\,\cdot\,}$ the inner product and norm for scalar- or vector-valued functions in $\lt(\om)$. We introduce the notation $\scpltomw{\,\cdot\,}{\,\cdot\,}{\rho}:=\scp{\rho\,\cdot\,}{\,\cdot\,}_{\lt(\om)}$, which induces $\normltomw{\,\cdot\,}{\rho}$, where $\rho$ belongs to the space of essentially bounded functions $\li(\om)$. If $\rho$ is self-adjoint and uniformly positive definite, they become an inner product and a norm in $\lt(\om)$, respectively. The space of scalar-valued functions in $\lt(\om)$ with zero mean is defined as
\begin{equation*}
	\ltz(\om) := \bset{\varphi\in\lt(\om)}{ \int_\om \varphi \, \dx = 0 } ,
\end{equation*}
and as usual, for a vector-valued function $\phi$ we write $\phi\in\ltz(\om)$ if all its components belong to $\ltz(\om)$. In the rest of this paper we may drop $\om$ in our notations for brevity, i.e., $\lt := \lt(\om)$.

We define the usual Sobolev spaces
\begin{align*}
	\ho	& := \set{\varphi\in\lt}{\na\varphi\in\lt} ,
				& \hoc & := \ol{\cic}^{\ho} , \\
	\d	& := \set{\phi\in\lt}{\div\phi\in\lt} ,
				& \dc & := \ol{\cic}^{\d} ,
\end{align*}
which are Hilbert spaces. Note that on the former spaces the differential operators are now defined in the usual weak sense. The latter spaces, where the closures are taken with respect to graph norms, generalize the classical homogeneous scalar and normal boundary conditions, respectively.

The Friedrichs inequality reads as
\begin{equation*}
	\forall \varphi \in \hoc \qquad
	\normltom{\varphi} \leq \cf \normltom{\na\varphi} ,
\end{equation*}
where $\cf=\cf(\om)>0$ is called the Friedrichs constant. Note that $\cf$ is assumed to be the best possible, i.e., smallest possible constant for which the Friedrichs inequality holds. A commonly utilized upper bound for $\cf$ is \cite{mikhlinbook}
\begin{equation} \label{eq:cf_mikhlin}
	\cf \leq \brac{\pi \sqrt{\frac{1}{l_1^2}+\cdots+\frac{1}{l_d^2}}}^{-1} .
\end{equation}

This note is dedicated to finding upper bounds for the constant $\cfA=\cfA(\om,\A)>0$ in the weighted Friedrichs type inequality
\begin{equation} \label{eq:cfA}
	\forall \varphi \in \hoc \qquad
	\normltom{\varphi} \leq \cfA \normltomw{\na \varphi}{\A}
\end{equation}
for bounded $\om$. Here $\A\in\li$ is a self-adjoint (i.e., equal to its conjugate transpose), uniformly positive definite matrix valued function $\A:\om\to\reals^{d \times d}$, i.e., it satisfies
\begin{equation} \label{eq:Acond1}
	\exists \ul{\A}>0 \qquad 
	\forall\phi\in\lt \qquad
	\ul{\A} \normltom{\phi}^2
	\leq \scpltom{\A\phi}{\phi} .
\end{equation}
Estimates for $\cfA$ can be calculated by using estimates for $\cf$, since obviously
\begin{equation*}
	\normltom{\varphi}
	\leq \cf \normltom{\na\varphi}
	\leq \frac{\cf}{\sqrt{\ul{\A}}} \normltomw{\na\varphi}{\A}
\end{equation*}
holds. Note that since the first estimation step is done using the Friedrichs inequality, and contains the full gradient on the right hand side, it is inevitable that the final estimation step involves a division with the smallest eigenvalue of $\A$. Estimating further by using \eqref{eq:cf_mikhlin} we obtain the estimate
\begin{equation} \label{eq:coarse}
	\cfA \leq \brac{\pi \sqrt{\ul{\A}\brac{\frac{1}{l_1^2}+\cdots+\frac{1}{l_d^2}}}}^{-1} ,
\end{equation}
which blows up as $\ul{\A}$ approaches zero.

Having computable upper bounds of Friedrichs, Poincar\'e, and Maxwell type constants related to both weighted and non-weighted variants of corresponding inequalities is important in a posteriori error estimation. Error upper bounds typically contain these constants, and are especially important for functional type error estimates, where \emph{guaranteed} upper bounds of the exact error are desired. In this note we omit a literature overview of a posteriori error estimation, and instead refer the reader to the books\cite{NeittaanmakiRepin2004,repinbookone,MaliRepinNeittaanmaki2014,ainsworthodenbookone,verfurthbook2013}.

Some references with upper bounds of Friedrichs and Poincar\'e type constants are the book \cite{mikhlinbook} and the paper \cite{payneweinbergerpoincareconvex} (see also \cite{bebendorfpoincareconvex}). We also cite the interesting survey article \cite{KuznetsovNazarov2015}. Some more recent work on the subject include \cite{vejchodskyfriedrichs2012}, where a weighted Friedrichs inequality similar to \eqref{eq:cfA} is considered. The author calculates numerically two-sided bounds of a Friedrichs type constant in weighted norms. This approach allows for mixed boundary conditions. In \cite{repinpoincare2012} Friedrichs and Poincar\'e inequalities in non-weighted norms with mixed boundary conditions are considered. This approach involves decomposing the domain into smaller subdomains for which Friedrichs and Poincar\'e constants are known. The resulting upper bounds depend on the decomposition. Computable upper bounds of Maxwell constants for convex domains have been studied by the authors of the present note. In the second author's papers \cite{paulymaxconst0,paulymaxconst1,paulymaxconst2} it is shown that the Maxwell constants are bounded by above by the Poincar\'e constant, and a small improvement to these results can be found in the first author's paper \cite{anjamhelmholtz}.

In this note we show that in the case of full homogeneous Dirichlet boundary conditions, there is a simple way to obtain an upper bound of $\cfA$ with better properties than \eqref{eq:coarse}. The upper bound, derived in Section \ref{sec:CF}, follows from well known results. In this section we also demonstrate the benefit of using the improved upper bound of $\cfA$ by a numerical example where we perform a posteriori error estimation of an elliptic problem. In Section \ref{sec:MC} we use this upper bound to improve an upper bound of the tangential Maxwell constant for convex domains in $\reals^3$.


\section{A Weighted Friedrichs Inequality}
\label{sec:CF}

The calculations of this section are based on the well known one-dimensional inequality
\begin{equation} \label{eq:onedim}
	\forall \varphi\in\hoc((0,l)) \qquad
	\int_0^l \abs{\varphi(y)}^2 \dy
	\leq \frac{l^2}{\pi^2} \int_0^l \abs{\varphi'(y)}^2 \dy ,
\end{equation}
where $0<l<\infty$. Using this inequality one can proof a Friedrichs type inequality involving only one partial derivative, and by an additional estimation step obtain an inequality involving the full gradient. In the case of bounded domains, this would result in the estimate \eqref{eq:cf_mikhlin}. However, we will need the intermediate result involving only one partial derivative. Note, that since we want to control all partial derivatives separately (with respect to the \emph{already chosen} coordinate system), we cannot rotate the domain.

\begin{lem} \label{lem:PART}
Let $\om$ be bounded, and $i\in\{1,\ldots,d\}$. Then we have the estimate
\begin{equation*}
	\forall \varphi\in\hoc \qquad
	\normltom{\varphi} \leq
	\frac{l_i}{\pi} \normltom{\p_i\varphi} .
\end{equation*}
\end{lem}

\begin{proof}
Consider first the real valued case and $i=1$. For any $\varphi\in\cic(\om)$ its zero-extension $\hat{\varphi} : I \to \reals$ belongs to $\cic(I)$. For any $\xt:=(x_2,\ldots,x_d)$ belonging to $\tilde{I} := (0,l_2)\times\cdots\times(0,l_d)$, the function $\hat{\varphi}(x_1,\xt)$ is a real valued function of one variable vanishing at the endpoints of the interval $(0,l_1)$, so by \eqref{eq:onedim} we have
\begin{equation*}
	\int_0^{l_1} \abs{\hat{\varphi}(x_1,\xt)}^2 \dx_1
	\leq \frac{l_1^2}{\pi^2} \int_0^{l_1} \abs{\p_1\hat{\varphi}(x_1,\xt)}^2 \dx_1 .
\end{equation*}
By integrating the above with respect to $\xt$ in $\tilde{I}$, we obtain
\begin{equation*}
	\norm{\hat{\varphi}}_{\lt(I)}^2
	\leq \frac{l_1^2}{\pi^2} \norm{\p_1\hat{\varphi}}_{\lt(I)}^2
	\qquad \Rightarrow \qquad
	\norm{\varphi}_{\lt(\om)}^2
	\leq \frac{l_1^2}{\pi^2} \norm{\p_1\varphi}_{\lt(\om)}^2	,
\end{equation*}
since the norms are nonzero only in $\om$. By density the above holds for any $\varphi\in\hoc(\om)$. By an identical procedure the assertion follows for $i\in\{2,\ldots,d\}$ for real valued functions. Having established the assertion for real valued functions, it is clear that it holds for complex valued functions as well.
\end{proof}

We now consider the constant $\cfA$ in the inequality \eqref{eq:cfA}. We assume that $\A\in\li$ is a self-adjoint diagonal matrix
\begin{equation} \label{eq:Adef}
	\A :=
	\begin{pmatrix}
	\A_1 	& 			& 0		\\
	     	& \ddots		&		\\
	0		&			& \A_d
	\end{pmatrix}
\end{equation}
satisfying uniform positive definiteness \eqref{eq:Acond1}, which in this case is equivalent to
\begin{equation} \label{eq:Acond2}
	\exists \ul{\A}_i > 0 \qquad
	\forall \varphi\in\lt \qquad 
	\ul{\A}_i \normltom{\varphi}^2 \leq \scpltom{\A_i\varphi}{\varphi} , \qquad
	i = 1,\ldots,d .
\end{equation}
Note that such an $\alpha$ has no imaginary part, and that $\ul{\A} = \min\{\ul{\A}_1,\ldots,\ul{\A}_d\}$.

\begin{theo} \label{thm:A}
Let $\om$ be bounded and $\A\in\li$ be a self-adjoint diagonal matrix satisfying \eqref{eq:Adef}--\eqref{eq:Acond2}. Then we have the estimate
\begin{equation*}
	\cfA \leq
	\brac{\pi\sqrt{\frac{\ul{\A}_1}{l_1^2}+\cdots+\frac{\ul{\A}_d}{l_d^2}}}^{-1} .
\end{equation*}
\end{theo}

\begin{proof}
Let $\varphi\in\hoc$. Since $\A$ is diagonal, the weighted norm can be written as
\begin{equation*}
	\normltomw{\na\varphi}{\A}^2
	= \normltomw{\p_1\varphi}{\A_1}^2 + \cdots + \normltomw{\p_d\varphi}{\A_d}^2 .
\end{equation*}
Lemma \ref{lem:PART} gives
\begin{equation*}
	\normltom{\varphi}^2
	\leq \frac{l_i^2}{\pi^2} \normltom{\p_i\varphi}^2
	\leq \frac{l_i^2}{\pi^2\ul{\A}_i} \normltomw{\p_i\varphi}{\A_i}^2
\end{equation*}
for any $i\in\{1,\ldots,d\}$. By multiplying the above by $\ul{\A}_i/l_i^2$ and summing up the $d$ inequalities, we obtain
\begin{equation*}
	\brac{\frac{\ul{\A}_1}{l_1^2}+\cdots+\frac{\ul{\A}_d}{l_d^2}} \normltom{\varphi}^2
	\leq \frac{1}{\pi^2} \normltomw{\na\varphi}{\A}^2 ,
\end{equation*}
which implies the assertion.
\end{proof}

\begin{rem} \label{rem:A}
\mbox{}
\begin{itemize}
\item[\bf(i)] Theorem \ref{thm:A} with $\A=\id$ results in the estimate \eqref{eq:cf_mikhlin}.
\item[\bf(ii)] It is easy to see that the upper bound of Theorem \ref{thm:A} is always smaller or equal to the upper bound \eqref{eq:coarse}.
\item[\bf(iii)] The above procedure furnishes upper bounds of $\cfA$ even when the diagonal matrix $\A$ is not uniformly positive definite (see Appendix \ref{app:semidef}).
\item[\bf(iv)] The result of Lemma \ref{lem:PART}, and thus also Theorem \ref{thm:A}, holds also for an unbounded domain lying between two parallel hyperplanes, provided that the hyperplanes are not parallel to \emph{any} coordinate axes. The reason for this limitation is because in the proof of Lemma \ref{lem:PART} we need \eqref{eq:onedim} in the direction of all the coordinate axes. Note that in this case the constants $l_i$ denote the distance of these two hyperplanes measured by a line parallel to the $x_i$ axis. The proof of this result is only slightly more involved.
\item[\bf(v)] An upper bound similar to Theorem \ref{thm:A} for bounded domains can be obtained provided that the homogeneous Dirichlet boundary conditions hold in at least one direction. I.e., for the unit square in $\reals^2$, it is enough that one pair of opposing boundaries have the boundary condition; if this pair is the one parallel to the $x_2$-axis, then the boundary condition is in the direction of the $x_1$-axis, and we have $\cfA \leq (\pi\sqrt{\ul{\alpha}_1/l_1^2})^{-1}$. However, since the opposing boundary parts must have a "straight line of sight" to each other, the possible domains are quite limited. More variety in domains is achieved, if one uses (instead of \eqref{eq:onedim})
\begin{equation*}
	\int_0^l \abs{\varphi(y)}^2 \dy
	\leq \frac{l^2}{2} \int_0^l \abs{\varphi'(y)}^2 \dy ,
\end{equation*}
which holds for all functions $\varphi\in\ho((0,l))$ vanishing either on the beginning  or the end of the interval. In this way only one of the opposing boundaries need to have the boundary condition. For this inequality see, e.g., \cite[p. 158]{adamsbook}.
\end{itemize}
\end{rem}

Under certain conditions non-diagonal $\A$ can be handled as well. For readability we consider only the three dimensional case. For any self-adjoint $\A(x)=\{\A_{ij}(x)\}_{i,j=1}^3$ from $\li$ we define
\begin{equation} \label{eq:At}
	\At :=
	\begin{pmatrix}
		\At_1 	& 0		& 0		\\
		0     	& \At_2	& 0		\\
		0		& 0		& \At_3
	\end{pmatrix} ,
	\qquad
	\begin{matrix}
		\At_1 := \A_{11} - (\abs{\Re\A_{12}} + \abs{\Im\A_{12}} + \abs{\Re\A_{13}} + \abs{\Im\A_{13}}) , \\
		\At_2 := \A_{22} - (\abs{\Re\A_{12}} + \abs{\Im\A_{12}} + \abs{\Re\A_{23}} + \abs{\Im\A_{23}}) , \\
		\At_3 := \A_{33} - (\abs{\Re\A_{13}} + \abs{\Im\A_{13}} + \abs{\Re\A_{23}} + \abs{\Im\A_{23}}) .
	\end{matrix}
\end{equation}
It is easy to verify that $\At$ is self-adjoint, and that
\begin{equation} \label{eq:lowerbound}
	\forall\phi\in\lt \qquad
	\normltomw{\phi}{\At} \leq \normltomw{\phi}{\A}
\end{equation}
holds. If $\At$ is also uniformly positive definite, i.e., it satisfies
\begin{equation*}
	\exists \ul{\At}_i > 0 \qquad
	\forall \varphi\in\lt \qquad 
	\ul{\At}_i \normltom{\varphi}^2 \leq \scpltom{\At_i\varphi}{\varphi} , \qquad
	i = 1,2,3 ,
\end{equation*}
we can directly use Theorem \ref{thm:A} to obtain an estimate of $\cfA$.

\begin{theo} \label{thm:A2}
Let $\om\subset\reals^3$ be bounded and $\A\in\li$ be a self-adjoint matrix valued function for which $\At$, defined by \eqref{eq:At}, is uniformly positive definite. Then we have the estimate
\begin{equation*}
	\cfA \leq
	\brac{\pi\sqrt{\frac{\ul{\tilde{\A}}_1}{l_1^2}
		+ \frac{\ul{\tilde{\A}}_2}{l_2^2}
		+ \frac{\ul{\tilde{\A}}_3}{l_3^2}}}^{-1} .
\end{equation*}
\end{theo}

\begin{proof}
Let $\varphi\in\hoc$. Theorem \ref{thm:A} gives
\begin{equation*}
	\normltom{\varphi}
	\leq \brac{\pi\sqrt{\frac{\ul{\tilde{\A}}_1}{l_1^2}
		+ \frac{\ul{\tilde{\A}}_2}{l_2^2}
		+ \frac{\ul{\tilde{\A}}_3}{l_3^2}}}^{-1}
		\normltomw{\na\varphi}{\At} ,
\end{equation*}
and with \eqref{eq:lowerbound} we have the assertion.
\end{proof}

\begin{rem} \label{rem:A2}
For $\At$ to be uniformly positive definite would require that the off-diagonal entries of $\A$ be comparatively small compared to its diagonal entries. However, now Remark \ref{rem:A} (iii) holds with respect to $\At$. In particular, for an upper bound of $\cfA$ it is enough that $\At$ is positive semi-definite such that \emph{one} of the diagonal entries of $\At$ is uniformly positive definite.
\end{rem}

We demonstrate the derived results in the real valued setting through some examples.

\begin{ex}[Diagonal matrix $\A$] \label{ex:calpha}
Let $\om\subset(0,1)^2$ and $\A$ be the uniformly positive definite constant matrix
\begin{equation*}
	\A = \begin{pmatrix} 1 & 0 \\ 0 & \delta \end{pmatrix} , \qquad \delta>0 .
\end{equation*}
The estimate \eqref{eq:coarse} gives the upper bound
\begin{equation} \label{eq:calpha1}
	\cfA \leq \big( \pi\sqrt{2\min\{1,\delta\}} \big)^{-1} ,
\end{equation}
and Theorem \ref{thm:A} gives
\begin{equation} \label{eq:calpha2}
	\cfA \leq
	\big( \pi\sqrt{1+\delta} \big)^{-1} .
\end{equation}
It is easy to see that the latter does not blow up as $\delta$ becomes smaller. Table \ref{tbl:calpha} shows the values of the bounds with different $\delta$.
\end{ex}

\begin{table}[h]
\caption{Example \ref{ex:calpha}: Values of the upper bounds \eqref{eq:calpha1} and \eqref{eq:calpha2} with different $\delta$.}
\begin{tabular}{c|r|r|r|r|r|r|r}
$\delta$ & \multicolumn{1}{c|}{$10^{-6}$} & \multicolumn{1}{c|}{$10^{-4}$} & \multicolumn{1}{c|}{$10^{-2}$} & \multicolumn{1}{c|}{$1$} & \multicolumn{1}{c|}{$10^{2}$} & \multicolumn{1}{c|}{$10^{4}$} & \multicolumn{1}{c}{$10^{6}$} \\ 
\hline
\eqref{eq:calpha1} & 225.07908 & 22.50791 & 2.25079 & 0.22508 & 0.22508 & 0.22508 & 0.22508 \\
\eqref{eq:calpha2} & 0.31831 & 0.31829 & 0.31673 & 0.22508 & 0.03167 & 0.00318 & 0.00032
\end{tabular}
\label{tbl:calpha}
\end{table}

\begin{ex}[Solution theory for a reaction-diffusion problem] \label{ex:solution}
Consider the following reaction-diffusion problem: find $u\in\hoc$ satisfying
\begin{equation*}
	-\div\A\na u + \rho\,u = f ,
\end{equation*}
where $f\in\lt$, $\rho\in\reals$, and $\A\in\li$ is a symmetric uniformly positive definite matrix valued function. The variational formulation of this problem reads as
\begin{equation*}
	\forall \varphi\in\hoc \qquad
	\scpltomw{\na u}{\na \varphi}{\A}
	+ \scpltomw{u}{\varphi}{\rho}
	= \scpltom{f}{\varphi} .
\end{equation*}
By setting $\varphi=u$ in the bilinear form on the left hand side, we obtain
\begin{align*}
	\scpltomw{\na u}{\na u}{\A} + \scpltomw{u}{u}{\rho}
	& = (1-\epsilon) \normltomw{\na u}{\A}^2
		+ \epsilon \normltomw{\na u}{\A}^2
		+ \normltomw{u}{\rho}^2 \\
	& \ge (1-\epsilon) \ul{\A} \normltom{\na u}^2
		+ \bigg( \frac{\epsilon}{\cfA^2}+\rho \bigg) \normltom{u}^2 ,
\end{align*}
where $0<\epsilon<1$. We observe that this form is coercive provided that
\begin{equation*}
	\frac{\epsilon}{\cfA^2}+\rho > 0
\end{equation*}
holds, and under this condition a unique solution exists in $\hoc$ by the Riesz representation theorem. Let $\om\subset(0,1)^2$ and
\begin{equation*}
	\A = \begin{pmatrix} 1 & 0 \\ 0 & 100 \end{pmatrix} .
\end{equation*}
Using \eqref{eq:coarse} to estimate $\cfA$ (see Example \ref{ex:calpha}), we see that for existence and uniqueness of a solution, the necessary condition is $\rho > - \epsilon\,2\pi^2$, but using Theorem \ref{thm:A} the necessary condition becomes $\rho > - \epsilon\,101\pi^2$, allowing for a larger range of admissible $\rho$.
\end{ex}

\begin{ex}[Non-diagonal matrix $\A$] \label{ex:calpha2}
Let $\om\subset(0,1)^3$, and
\begin{equation*}
	\A = \begin{pmatrix} 3 & 1 & 1 \\ 1 & 300 & 1 \\ 1 & 1 & 3 \end{pmatrix} ,
	\qquad
	\At = \begin{pmatrix} 1 & 0 & 0 \\ 0 & 298 & 0 \\ 0 & 0 & 1 \end{pmatrix} ,
\end{equation*}
where $\At$ is calculated according to \eqref{eq:At}. Now $\ul{\A}=2$, and \eqref{eq:coarse} gives the bound $\cfA \leq \brac{\pi\sqrt{6}}^{-1}$. Theorem \ref{thm:A2} gives the upper bound $\cfA \leq \brac{\pi\sqrt{300}}^{-1}$, which is sharper.
\end{ex}

As stated in the introduction, the motivation for deriving computable upper bounds for the constant $\cfA$ is that it is essential in a posteriori error estimation for numerical approximations of elliptic partial differential equations. As an example we consider the diffusion problem in the real valued setting, in a bounded domain $\om$, with homogeneous Dirichlet boundary conditions on the whole boundary: find $u\in\hoc$ satisfying
\begin{equation*}
	- \div \A\na u = f ,
\end{equation*}
where $f\in\lt$ and $\A\in\li$ is a symmetric uniformly positive definite matrix valued function. The variational formulation for this problem reads as
\begin{equation} \label{eq:var}
	\forall \varphi\in\hoc \qquad
	\scpltomw{\na u}{\na \varphi}{\A} = \scpltom{f}{\varphi} .
\end{equation}
Since \eqref{eq:cfA} is satisfied, a unique solution $u\in\hoc$ exists by the Riesz representation theorem. By setting $\varphi=u$ in \eqref{eq:var} we see that the solution depends continuously on the right hand side:
\begin{equation*}
	\normltomw{\na u}{\A}^2
	= \scpltom{f}{u}
	\leq \normltom{f}\normltom{u}
	\leq \cfA\normltom{f} \normltomw{\na u}{\A}
	\quad
	\Rightarrow
	\quad
	\normltomw{\na u}{\A} \leq \cfA \normltom{f} .
\end{equation*}

We now present the functional type a posteriori error upper bound, which can be found in, e.g., the books \cite{NeittaanmakiRepin2004,repinbookone,MaliRepinNeittaanmaki2014}.

\begin{theo} \label{thm:apost}
Let $\ut\in\hoc$ be an arbitrary approximation of $u$, and $\cfAt$ be any approximation of $\cfA$ from above. Then we have the estimate
\begin{equation*}
	\forall y\in\d \qquad
	\normltomw{\na(u-\ut)}{\A} \leq \cfAt \normltom{f+\div y}
	+ \normltomwi{y-\A\na\ut}{\A} := M(\cfAt,\ut,y) .
\end{equation*}
\end{theo}

\begin{proof}
We begin by subtracting the term $\scpltomw{\na\ut}{\na\varphi}{\A}$ from both sides of \eqref{eq:var} and obtain
\begin{align*}
	\scpltomw{\na(u-\ut)}{\na \varphi}{\A}
	& = \scpltom{f}{\varphi} - \scpltomw{\na\ut}{\na \varphi}{\A} \\
	& = \scpltom{f+\div y}{\varphi} + \scpltom{y-\A\na\ut}{\na\varphi} \\
	& \leq \normltom{f+\div y} \normltom{\varphi}
		+ \normltomwi{y-\A\na\ut}{\A} \normltomw{\na\varphi}{\A} \\
	& \leq \brac{ \cfA \normltom{f+\div y} + \normltomwi{y-\A\na\ut}{\A} }
		\normltomw{\na\varphi}{\A} ,
\end{align*}
where we used $\scpltom{\div y}{\varphi} + \scpltom{y}{\na\varphi} = 0$ and \eqref{eq:cfA}. Setting $\varphi=u-\ut$ finishes the proof.
\end{proof}

\begin{rem}
By using the upper bound $\cfA \leq \cf/\sqrt{\ul{\A}}$ for the value of $\cfAt$ we obtain the most commonly used form of this functional type a posteriori error upper bound for the diffusion problem.
\end{rem}

Note that the above estimate is sharp, i.e., theoretically there is no gap between the exact error and the estimate. This is seen by setting $y=\A\na u\in\d$. The first term of the error functional $M$ vanishes, and it becomes apparent that sharpness does not depend on $\cfA$. However, obtaining good error bounds requires not only choosing $y$ close to the exact flux $\A\na u$, but also having good upper bounds for the unknown constant $\cfA$. Especially in the case when $-\div y$ is not close to $f$, a large over-estimation of the constant $\cfA$ will lead to a large over-estimation of the error, as we will now demonstrate.

\begin{ex}[Error estimation with Raviart-Thomas averaging] \label{ex:apost_avg}
We solve the diffusion problem \eqref{eq:var} in in the L-shaped domain $\om=(0,1)^2\setminus[(1/2,1)\times(0,1/2)]$ with
\begin{equation*}
	\A = \begin{pmatrix} 1 & 0 \\ 0 & 10^{-4} \end{pmatrix} ,
	\qquad
	f = 1 .
\end{equation*}
We use linear nodal finite elements in triangles to solve \eqref{eq:var}, and denote the approximation by $\ut$. The function $y$ in the functional $M$ is obtained by averaging $\A\na\ut$ to the edges of the mesh resulting in a function from the linear Raviart-Thomas finite element space, which is a subspace of $\d$. We denote this averaging operator by $G_{\rm RT}$. Using \eqref{eq:coarse} to estimate the value of $\cfA$ (see Example \ref{ex:calpha}), we have the estimate
\begin{equation} \label{eq:apost_avg1}
	\normltomw{\na(u-\ut)}{\A} \leq M(22.50791,\ut,G_{\rm RT}(\A\na\ut)) ,
\end{equation}
and by using Theorem \ref{thm:A} we obtain the estimate
\begin{equation} \label{eq:apost_avg2}
	\normltomw{\na(u-\ut)}{\A} \leq M(0.31829,\ut,G_{\rm RT}(\A\na\ut)) .
\end{equation}
Since $- \div G_{\rm RT}(\A\na\ut)$ is only a rough approximation of $f$, the quality of the latter estimate is better, as is seen from Table \ref{tbl:apost_avg}.
\end{ex}

\begin{table}[h]
\caption{Example \ref{ex:apost_avg}: Values of the upper bounds \eqref{eq:apost_avg1} and \eqref{eq:apost_avg2} with different meshes.}
\begin{tabular}{c|r|r|r|r|r}
\#elements & \multicolumn{1}{c|}{$384$} & \multicolumn{1}{c|}{$1536$} & \multicolumn{1}{c|}{$6144$} & \multicolumn{1}{c|}{$24576$} & \multicolumn{1}{c}{$98304$} \\ 
\hline
\eqref{eq:apost_avg1} & 18.4444 & 17.1419 & 16.1891 & 14.9832 & 13.2664 \\
\eqref{eq:apost_avg2} & 1.5563 & 0.9166 & 0.5705 & 0.3809 & 0.2695 \\
\end{tabular}
\label{tbl:apost_avg}
\end{table}


\section{The Tangential Maxwell Inequality for Convex Domains in $\reals^3$}
\label{sec:MC}

In this section, after introducing some additional notation, we improve an upper bound of the tangential Maxwell constant using Theorem \ref{thm:A}. Throughout, we work in three dimensions, i.e., the domain $\om$ belongs to $\reals^3$.

Aside from the Sobolev spaces already defined in the introduction, we also define
\begin{align*}
	\r	& := \set{\phi\in\lt}{\rot\phi\in\lt} ,
				& \rc & := \ol{\cic}^{\r} ,
\end{align*}
which are Hilbert spaces. As before, on the former space the differential operator $\rot$ is defined in the usual weak sense. The latter space, where the closure is taken with respect to the graph norm, generalizes the classical tangential boundary condition. Note that the rotation $\rot$ is often written as $\curl$ or $\na\times$ in the literature.

Let $\eps\in\li$ be a self-adjoint uniformly positive definite function $\eps:\om\to\reals^{3\times3}$, i.e., it satisfies
\begin{equation*}
	\exists \ul{\eps},\ol{\eps} > 0 \qquad
	\forall\phi\in\lt \qquad 
	\ul{\eps} \normltom{\phi}^2 \leq
	\scpltom{\eps\phi}{\phi} \leq
	\ol{\eps} \normltom{\phi}^2 .
\end{equation*}
Note that here the overline in $\ol{\eps}$ does not denote complex conjugation. In this section the properties assumed from $\eps$ are similar to what was assumed for $\A$ in the previous section. We use $\eps$ instead of $\A$ to conform to the usual notation used in electromagnetic theory where $\eps$ denotes the electric permittivity of the media.

Since $\om$ is convex, it is also Lipschitz \cite{grisvard1985}, and Rellich's selection theorem and Weck's selection theorem \cite{weck} hold. Thus the spaces in the well known Helmholtz decompositions (see, e.g., \cite{leisbook})
\begin{align}
	\lt & = \na\hoc \opluse \inv{\eps} \dz
			= \rcz \opluse \inv{\eps} \rot\r ,
			\qquad \na\hoc = \rcz , \qquad \dz = \rot\r, \label{eq:hdconvex1} \\
	\lt & = \na\ho \opluse \inv{\eps} \dcz
			= \rz \opluse \inv{\eps} \rot\rc ,
			\qquad \na\ho = \rz , \qquad \dcz = \rot\rc \label{eq:hdconvex2}
\end{align}
are closed. Here $\opluse$ denotes orthogonal sum with respect to the weighted scalar product $\scpltomw{\,\cdot\,}{\,\cdot\,}{\eps}$. If $\eps=\id$ in these decompositions, we omit it, i.e., we write $\oplus$ instead of $\oplus_\id$.

The Poincar\'e and tangential Maxwell estimates read as
\begin{align*}
	& \forall \varphi \in \ho \cap \ltz & &
	\normltom{\varphi} \leq \cp \normltom{\na\varphi} , \\
	& \forall \phi \in \rc \cap \inv{\eps}\d & &
	\normltomw{\phi}{\eps} \leq \cme \sqrt{ \normltom{\div\eps\phi}^2 + \normltom{\rot\phi}^2 } ,
\end{align*}
where $\cp=\cp(\om)>0$ is the Poincar\'e constant and $\cme=\cme(\om,\eps)>0$ the tangential Maxwell constant. For convex domains we have the computable upper bound
\begin{equation} \label{eq:poincarebound}
	\cp \leq \frac{\diam(\om)}{\pi}
\end{equation}
by Payne and Weinberger \cite{payneweinbergerpoincareconvex} (see also \cite{bebendorfpoincareconvex}). In \cite{paulymaxconst1} (see also \cite{paulymaxconst0,paulymaxconst2,anjamhelmholtz}) it was shown that
\begin{equation} \label{eq:maxbound1}
	\cme \leq \max\left\{ \frac{\cf}{\sqrt{\ul{\eps}}} , \sqrt{\ol{\eps}}\cp \right\} ,
\end{equation}
which, together with \eqref{eq:cf_mikhlin} and \eqref{eq:poincarebound}, gives the computable upper bound
\begin{equation} \label{eq:TMC_bound_coarse}
	\cme \leq \frac{1}{\pi} \max \Bigg\{ \sqrt{\ul{\eps}\brac{\frac{1}{l_1^2}+\frac{1}{l_2^2}+\frac{1}{l_3^2}}}^{-1}, \sqrt{\ol{\eps}} \diam(\om) \Bigg\} .
\end{equation}
As $\ul{\eps}$ approaches zero, this bound blows up, as does \eqref{eq:coarse}. In the following we improve this bound so that this does not happen. The rest of this section essentially follows the sequence found in \cite{paulymaxconst1}, with a few small differences to allow for the improved bound.

\begin{lem} \label{lem:horeg}
Let $\om\subset\reals^3$ be convex and
$\phi\in\rc\cap\d$ or
$\phi\in\r\cap\,\dc$. Then $\phi\in\ho$ and
$\normltom{\na\phi}^2 \leq \normltom{\div\phi}^2 + \normltom{\rot\phi}^2$.
\end{lem}

\begin{proof}
See \cite[Thm. 2.17]{amrouche98}.
\end{proof}

\begin{lem} \label{lem:zeromean}
Let $\om\subset\reals^d$. The inclusion $\dcz \subset \ltz$ holds.
\end{lem}

\begin{proof}
Let $\phi\in\dcz$. Then
$\scpltom{\phi_i}{1} = \scpltom{\phi}{\na x_i} = -\scpltom{\div\phi}{x_i} = 0$
for any $i\in\{1,\ldots,d\}$.
\end{proof}

\begin{lem} \label{lem:TMC}
Let $\om\subset\reals^3$ be convex, $\phi\in\rcz\cap\inv{\eps}\d$ and $\psi\in\rc\cap\inv{\eps}\dz$. Then the estimates $\normltomw{\phi}{\eps} \leq \cfe \normltom{\div\eps\phi}$ and $\normltomw{\psi}{\eps} \leq \sqrt{\ol{\eps}}\cp \normltom{\rot\psi}$ hold.
\end{lem}

\begin{proof}
We prove the first estimate following \cite{paulymaxconst1}. Let $\phi$ belong to $\rcz\cap\inv{\eps}\d$. Since $\na\hoc = \rcz$, there exists a scalar potential $\varphi\in\hoc$ such that $\phi=\na\varphi$, and we have
\begin{equation*}
	\normltomw{\phi}{\eps}^2
	= \scpltom{\eps\phi}{\phi}
	= \scpltom{\eps\phi}{\na\varphi}
	= - \scpltom{\div\eps\phi}{\varphi}
	\leq \normltom{\div\eps\phi} \normltom{\varphi}
	\leq \cfe \normltom{\div\eps\phi} \normltomw{\na\varphi}{\eps} ,
\end{equation*}
where we have applied \eqref{eq:cfA} with $\alpha=\eps$. This proves the first estimate. The second estimate we prove in the way presented in \cite{anjamhelmholtz}. Let $\psi$ belong to $\rc\cap\inv{\eps}\dz$. From \eqref{eq:hdconvex1}--\eqref{eq:hdconvex2} we deduce $\dz = \rot\r = \rot(\r\cap\dcz)$. The latter identity is seen by decomposing $\r$ using $\lt = \na\ho \oplus \dcz$. Thus, since $\eps\psi\in\dz$, there exists a vector potential $\xi\in\r\cap\dcz$ such that $\eps\psi=\rot\xi$. With Lemmas \ref{lem:horeg} and \ref{lem:zeromean} we see that $\xi$ also belongs to $\ho\cap\ltz$, and we can write
\begin{align*}
	\normltomw{\psi}{\eps}^2
	& = \scpltom{\eps\psi}{\psi}
	  = \scpltom{\rot\xi}{\psi} 
	  = \scpltom{\xi}{\rot\psi}
	  \leq \normltom{\xi} \normltom{\rot\psi}
	  \leq \cp \normltom{\na\xi} \normltom{\rot\psi} \\
	& = \cp \normltom{\rot\xi} \normltom{\rot\psi}
	  = \cp \normltom{\eps\psi} \normltom{\rot\psi}
	  \leq \sqrt{\ol{\eps}} \cp \normltomw{\psi}{\eps} \normltom{\rot\psi} ,
\end{align*}
where we used the Poincar\'e estimate and Lemma \ref{lem:horeg}. This proves the second estimate.
\end{proof}

\begin{rem}
The second estimate in Lemma \ref{lem:TMC} was proven with a slightly better constant in \cite{anjamhelmholtz}. This was achieved by refining the global zero mean property of Lemma $\ref{lem:zeromean}$ into a local equivalent. We skip this improvement in this note, since it does not change the overall behaviour of the forthcoming computable estimate with respect to small $\ul{\eps}$.
\end{rem}

We now sate the refinement of \eqref{eq:maxbound1}.

\begin{theo} \label{thm:TMC}
Let $\om\subset\reals^3$ be convex and $\eps\in\li$ be a self-adjoint uniformly positive definite matrix. Then $\cme \leq \max\{ \cfe, \sqrt{\ol{\eps}}\cp \}$ holds.
\end{theo}

\begin{proof}
We first use \eqref{eq:hdconvex1}--\eqref{eq:hdconvex2} to decompose the space $\rc\cap\inv{\eps}\d$.
By decomposing $\rc$ using $\lt = \rcz\opluse\inv{\eps}\dz$ we obtain $\rc = \rcz \opluse (\rc\cap\inv{\eps}\dz)$. Thus, we deduce
\begin{equation*}
	\rc\cap\inv{\eps}\d =
	(\rcz\cap\inv{\eps}\d) \opluse (\rc\cap\inv{\eps}\dz) .
\end{equation*}
Using the above decomposition we can write $\phi\in\rc\cap\inv{\eps}\d$ as the sum $\phi=\phi_1+\phi_2$, where
\begin{equation*}
	\phi_1 \in \rcz\cap\inv{\eps}\d, \qquad
	\phi_2 \in \rc\cap\inv{\eps}\dz,
\end{equation*}
and furthermore, $\div\eps\phi=\div\eps\phi_1$ and $\rot\phi=\rot\phi_2$ hold. By using Lemma \ref{lem:TMC} we then have
\begin{equation*}
	\normltomw{\phi}{\eps}^2
	= \normltomw{\phi_1}{\eps}^2 + \normltomw{\phi_2}{\eps}^2
	\leq \cfe^2 \normltom{\div\eps\phi_1}^2
		+ \ol{\eps}\cp^2 \normltom{\rot\phi_2}^2
	= \cfe^2 \normltom{\div\eps\phi}^2
		+ \ol{\eps}\cp^2 \normltom{\rot\phi}^2 ,
\end{equation*}
from which the assertion follows.
\end{proof}

The improvement of \eqref{eq:TMC_bound_coarse}, in the case of diagonal $\eps$, then looks like follows.

\begin{theo} \label{thm:TMC_bound}
Let $\om\subset\reals^3$ be convex and $\eps\in\li$ be a self-adjoint diagonal matrix such that properties \eqref{eq:Adef}--\eqref{eq:Acond2} with $\alpha=\eps$ hold. Then we have the estimate
\begin{equation*}
	\cme \leq \frac{1}{\pi} \max \Bigg\{ \sqrt{\brac{\frac{\ul{\eps}_1}{l_1^2}+\frac{\ul{\eps}_2}{l_2^2}+\frac{\ul{\eps}_3}{l_3^2}}}^{-1}, \sqrt{\ol{\eps}} \diam(\om) \Bigg\} .
\end{equation*}
\end{theo}

\begin{proof}
Theorem \ref{thm:TMC} together with Theorem \ref{thm:A} and \eqref{eq:poincarebound} result in the assertion.
\end{proof}

\begin{rem} \label{rem:TMC}
\mbox{}
\begin{itemize}
\item[\bf(i)] It is easy to see that the upper bound of Theorem \ref{thm:TMC_bound} is always smaller or equal to the upper bound \eqref{eq:TMC_bound_coarse}. Moreover, with $\eps$ a constant real number, these estimates coincide.
\item[\bf(ii)] The above procedure furnishes upper bounds of $\cme$ even when the diagonal matrix $\eps$ is not uniformly positive definite (see Appendix \ref{app:semidef}).
\end{itemize}
\end{rem}

Under certain conditions non-diagonal $\eps$ can be handled using Theorem \ref{thm:A2}. Below $\epst$ is the diagonal matrix related to $\eps$ satisfying the properties \eqref{eq:At}--\eqref{eq:lowerbound}.

\begin{theo} \label{thm:TMC_bound2}
Let $\om\subset\reals^3$ be convex and $\eps\in\li$ be a self-adjoint matrix valued function for which $\epst$ is uniformly positive definite. Then we have the estimate
\begin{equation*}
	\cme \leq \frac{1}{\pi} \max \Bigg\{ \sqrt{\brac{\frac{\ul{\epst}_1}{l_1^2}+\frac{\ul{\epst}_2}{l_2^2}+\frac{\ul{\epst}_3}{l_3^2}}}^{-1}, \sqrt{\ol{\eps}} \diam(\om) \Bigg\} .
\end{equation*}
\end{theo}

\begin{proof}
Theorem \ref{thm:TMC} together with Theorem \ref{thm:A2} and \eqref{eq:poincarebound} result in the assertion.
\end{proof}

\begin{rem}
For $\epst$ to be uniformly positive definite would require that the off-diagonal entries of $\eps$ be comparatively small compared to its diagonal entries. However, now Remark \ref{rem:TMC} (ii) holds with respect to $\epst$. In particular, for an upper bound of $\cme$ it is enough that $\epst$ is positive semi-definite such that \emph{one} of the diagonal entries of $\epst$ is uniformly positive definite.
\end{rem}

We conclude with an example.

\begin{ex}[Diagonal matrix $\eps$] \label{ex:ceps}
Let $\om\subset(0,1)^3$ and $\eps$ be the uniformly positive definite constant matrix
\begin{equation*}
	\eps = \begin{pmatrix} 1 & 0 & 0 \\ 0 & 1 & 0 \\ 0 & 0 & \delta \end{pmatrix} , \qquad \delta>0 .
\end{equation*}
The estimate \eqref{eq:TMC_bound_coarse} gives the upper bound
\begin{equation} \label{eq:ceps1}
	\cme \leq \frac{\sqrt{3}}{\pi} \max\left\{ \frac{1}{3\sqrt{\delta}} , 1 \right\} ,
\end{equation}
and Theorem \ref{thm:TMC_bound} gives
\begin{equation} \label{eq:ceps2}
	\cme \leq
	\frac{\sqrt{3}}{\pi} .
\end{equation}
Obviously the latter does not blow up as $\delta$ becomes smaller. Table \ref{tbl:ceps} shows the values of the bounds with different $\delta$.
\end{ex}

\begin{table}[h]
\caption{Example \ref{ex:ceps}: Values of the upper bounds \eqref{eq:ceps1} and \eqref{eq:ceps2} with different $\delta$.}
\begin{tabular}{c|r|r|r|r|r|r|r}
$\delta$ & \multicolumn{1}{c|}{$10^{-6}$} & \multicolumn{1}{c|}{$10^{-4}$} & \multicolumn{1}{c|}{$10^{-2}$} & \multicolumn{1}{c|}{$1$} & \multicolumn{1}{c|}{$10^{2}$} & \multicolumn{1}{c|}{$10^{4}$} & \multicolumn{1}{c}{$10^{6}$} \\ 
\hline
\eqref{eq:ceps1} & 183.77630 & 18.37763 & 1.83776 & 0.55133 & 0.55133 & 0.55133 & 0.55133 \\
\hline
\eqref{eq:ceps2} & \multicolumn{7}{c}{$\sqrt{3}/\pi \approx$ 0.55133}
\end{tabular}
\label{tbl:ceps}
\end{table}


\bibliographystyle{plain} 
\bibliography{biblio}


\appendix

\section{Positive Semi-Definite Matrices $\A$}
\label{app:semidef}

We shortly demonstrate that upper bounds for $\cfA$ can be obtained even if the self-adjoint $\A\in\li$ is only positive semi-definite: Let $\om\subset(-1,1)^3$ and $\A$ be the constant matrix
\begin{equation*}
	\A = \begin{pmatrix} 0 & 0 & 0 \\ 0 & 0 & 0 \\ 0 & 0 & 1 \end{pmatrix} .
\end{equation*}
The smallest eigenvalue is $\ul{\A}=0$, so the estimate \eqref{eq:coarse} cannot be used. However, as stated in Remark \ref{rem:A} (iii), we can still obtain upper bounds for $\cfA$; we can apply Lemma \ref{lem:PART} with $i=3$ to obtain
\begin{equation*}
	\normltom{\varphi}
	\leq \frac{2}{\pi} \normltom{\p_3\varphi}
	= \frac{2}{\pi} \normltomw{\na\varphi}{\A}
	\qquad \Rightarrow \qquad
	\cfA \leq \frac{2}{\pi} .
\end{equation*}
The quantity $\normltomw{\na\,\cdot\,}{\A}$ is now not a norm, though.

\end{document}